\documentclass[10pt]{amsart}
\usepackage{amsmath, amscd, amssymb, latexsym, hhline, euscript, mathrsfs}
\usepackage[all]{xy}
\usepackage{enumitem}
\usepackage[usenames]{color} % for editing
\def\a{{\alpha}}
\def\b{{\beta}}
\def\g{{\gamma}}
\def\e{{\epsilon}}

\def\k{{\Bbbk}}
\def\l{{\lambda}}

\def\Lam{{\Lambda}}

\def\BZ{{\mathbb{Z}}}

\def\OO{\mathfrak{O}}
\def\AA{\mathcal{A}}
\def\DD{\mathfrak{D}}
\def\CC{\mathcal{C}}
\def\CD{\mathcal{D}}

\def\CF{\mathcal{F}}
\def\RR{\mathcal{R}}

\def\ldu{{}^{\vee}}
\newcommand\ann{\operatorname{ann}}
\newcommand\Irr{\operatorname{Irr}}
\newcommand\Soc{\operatorname{Soc}}

\newcommand\ord{\operatorname{ord}}
\newcommand\lcm{\operatorname{lcm}}
\newcommand\Tr{\operatorname{Tr}}
\newcommand\Ext{\operatorname{Ext}}
\newcommand\id{\operatorname{id}}

\DeclareMathOperator\co{\operatorname{co}}
\def\to{\rightarrow}
\def\dim{{\mbox{\rm dim}}}
\def\ord{{\mbox{\rm ord}}}
\def\Hom{{\mbox{\rm Hom}}}

\def\Soc{{\mbox{\rm Soc}}}
\newcommand\enumeri[1]{\begin{enumerate}[label=\rm(\roman*), leftmargin=*] #1 \end{enumerate}}

\newcommand\C[1]{#1\mbox{-\bf{mod}}_{\operatorname{\mathsf{fin}}}}
\newcommand\CM[1]{#1\mbox{-\bf{mod}}}
\newcommand\ld[1]{{}_{#1}}

\newcommand\YD[1]{{}^{#1}_{#1}\!\mathcal{YD}}
\newcommand\inv{^{-1}}
\newcommand\bidu{^{\vee\vee}}

\newcommand\du{^{\vee}}
\newcommand\ol[1]{\overline{#1}}
\renewcommand\o{\otimes}
\newtheorem{mainthm}{Theorem}

\newtheorem{thm}{Theorem}[section]

\newtheorem{cor}[thm]{Corollary}
\newtheorem{prop}[thm]{Proposition}
\newtheorem{lem}[thm]{Lemma}

\newtheorem*{q}{Conjecture}
\newtheorem{remark}[thm]{Remark}

\theoremstyle{definition}

\numberwithin{equation}{section}

\makeatletter
\def\namelabel#1#2{\@bsphack
  \protected@write\@auxout{}%
         {\string\newlabel{#1.nme}{{#2}{#2}}}%
  \@esphack}

\makeatother

\title{On Hopf algebras of dimension $4p$}
\author{Yi-Lin Cheng}
\address{Department of Mathematics, Iowa State University, Ames, IA 50011, USA.}
 \email{elynn@iastate.edu}
\author{Siu-Hung Ng}
\address{Department of Mathematics, Iowa State University, Ames, IA 50011, USA.}
 \email{rng@iastate.edu}
 \thanks{The second author was partially supported by NSA grant H98230-08-1-0078 and NSF grant DMS10-01566.}
 \dedicatory{Dedicated to Professor Miriam Cohen on the occasion of her retirement}
\begin{document}
\maketitle
\begin{abstract}
In this paper, we prove that a non-semisimple Hopf algebra $H$ of dimension $4p$ with $p$ an odd prime over an algebraically closed field of characteristic zero is pointed provided $H$  contains more than two group-like elements. In particular, we prove that non-semisimple Hopf algebras of dimensions 20, 28 and 44 are pointed or their duals are pointed, and this completes the classification of Hopf algebras in these dimensions.
\end{abstract}
\section*{Introduction}
Let $\k$ be an algebraically closed field of characteristic zero. Kaplansky conjectured in \cite{Kapl75} that there are only finitely many isomorphism classes of finite-dimensional Hopf algebras over $\k$ for any given dimension. His conjecture was disproved by counterexamples of dimension $p^4$ with $p$  an odd prime (cf.  \cite{AS98}, \cite{BDG99} and \cite{Gel98}).
However, there are only finitely many isomorphism classes of dimension 16 \cite{GV}. The smallest known dimension for which infinitely many non-isomorphic Hopf algebras exist is 32 \cite{EG02}.

The dimension 30 has been recently classified in \cite{Fuk}. For the dimensions less than 32, only 20, 24, 27, 28 have not been completely classified, and they are products of at least three primes.  All other Hopf algebras of dimension less than 32 are known to have finitely many isomorphism classes (cf. \cite{B09} for more details).  This paper concerns mainly non-semisimple Hopf algebras of dimension $4p$, and we complete the classification for dimensions 20, 28 and 44 in Theorem \ref{t:2844}.

For a pair of distinct primes $p$ and $q$, semisimple Hopf algebras of dimension $pq^2 < 100$ are classified (cf. \cite{Na99}, \cite{Na01} and \cite{Na04}). By the recent result of \cite{ENO09},  these semisimple Hopf algebras are of Frobenius type. For $q=2$, it follows from  \cite[Theorem 0.1]{Na01} that there are exactly two isomorphism classes $\mathscr{A}_0, \mathscr{A}_1$ of non-trivial semisimple Hopf algebras over $\k$ of dimension $4p$. These semisimple Hopf algebras $\mathscr{A}_i$ were constructed in \cite{Gel97} and they are self-dual.

Non-semisimple pointed Hopf algebras of dimension $pq^2$ were classified in \cite[Lemma A.1]{AN1}. The classification can be described in terms of the Taft algebras and the Hopf algebra $\AA(\tau, \mu)$ which is defined as follows: for  a primitive $q$-th root of unity $\tau \in \k$, and  $\mu \in \{0, 1\}$,  $\AA(\tau, \mu)$ is the $\k$-algebra generated by the elements $a$ and $y$ with the relations
$$
a^{pq}=1, \quad y^q = \mu(1-a^q), \quad a y = \tau y a\,.
$$
The comultiplication $\Delta$ on $\AA(\tau, \mu)$ is defined by
$$
\Delta(a)=a \o a, \quad \Delta(y)= y \o 1 + a \o y\,.
$$
There are exactly $4(q-1)$ isomorphism classes of non-semisimple pointed Hopf algebras $H$ of dimension $pq^2$, and they are given by
$\AA(\tau, 0)$,  $\AA(\tau, 0)^*$, $\AA(\tau, 1)$, and $T_q \o \k[\BZ_p]$
where  $\tau$ is a primitive $q$-th root of unity, and $T_q$ is a Taft algebra of dimension $q^2$. The set of group-like elements of any of these Hopf algebras forms a cyclic group of order $pq$.

It was proved in \cite{Radf75} that $A(\tau, 1)^*$ is not a pointed Hopf algebra. Therefore,
\renewcommand{\theequation}{\dag}
\begin{equation}\label{eq:completelist}
\AA(\tau, 0),\quad \AA(\tau, 0)^*,\quad  \AA(\tau, 1), \quad \AA(\tau, 1)^* \quad \text{and}\quad T_q \o \k[\BZ_p]
\end{equation}
account for $5(q-1)$ isomorphism classes of non-semisimple Hopf algebras of dimension $pq^2$.

\renewcommand{\theequation}{\thesection.\arabic{equation}}

The classification of general non-semisimple Hopf algebras of dimension $pq^2$ remains open. The problem requires a thorough understanding of Hopf algebras of dimension $pq$ which is also open in general (cf. \cite{EG03}, \cite{Ng04} and \cite{Ng07}). However, the classification of Hopf algebras of dimension $p$, $p^2$, $2p$ (cf. \cite{Zhu}, \cite{Ng02} and \cite{Ng05}) suggests a possibility for dimension $2q^2$ or $4p$. It has been recently proved in \cite{HN} that a non-semisimple Hopf algebra of dimension $2q^2$ is isomorphic to exactly one in the list \eqref{eq:completelist} with $p=2$.

By the same rationale, one would expect the $4p$ case can be solved by adjusting the technique used in the $2p^2$ case.  In fact, we prove the following theorem which completes the classification of non-semisimple Hopf algebras of dimension $4p$ with more than two group-like elements.
\begin{mainthm}\label{t:Gorder}
Let $H$ be a non-semisimple Hopf algebra over $\k$ of dimension $4p$ where $p$ is an odd prime. Then $H$ is pointed if, and only if, $|G(H)|>2$.
\end{mainthm}
It remains unclear whether there exists a non-semisimple Hopf algebra $H$ of dimension $4p$ with fewer than three group-like elements such that neither $H$ nor $H^*$ is pointed.

    A $4p$-dimensional Hopf algebra which is neither pointed nor dual pointed could be constructed from a $p$-dimensional braided Hopf algebra $R$ in the Yetter-Drinfeld category $\YD{H_4}$ by Radford biproduct or bosonization $R\times H_4$, where $H_4$ denotes the Sweedler algebra. We prove in Proposition \ref{p:bosonization} that $H$ is such a biproduct if, and only if, both $H$ and $H^*$ admit a non-trivial skew primitive element.

For any odd prime $p \le 11$, we show in Lemma \ref{l:skewprimitive} by counting arguments that every non-semisimple Hopf algebra of dimension $4p$ admits a non-trivial skew primitive element. Together with the investigation of $p$-dimensional braided Hopf algebras in $\YD{H_4}$ in Sections \ref{s:braidedHopf} and \ref{s:dim 5}, we prove our main theorem in this paper:
\begin{mainthm}\label{t:2844} If $H$ is a non-semisimple Hopf algebra of dimension $4p$, where $p \le 11$ is an odd prime, then $H$ or $H^*$ is pointed. In particular, for each of these dimensions, there are exactly 5 isomorphism classes which are given by the list \eqref{eq:completelist}.
 \end{mainthm}
The non-semisimple case of dimension 12 was classified by Natale in \cite{Na02}, and the semisimple case was completed by Fukuda in \cite{Fu97}.

Theorem \ref{t:2844} and the main result of \cite{HN} justify the following conjecture.
\begin{q}
 Let $p, q$ be distinct primes. For any non-semisimple Hopf algebra $H$ over $\k$ of dimension $pq^2$,  $H$ or $H^*$ is pointed.
\end{q}
This paper is organized as follows: Section 1 collects some
notation and preliminary results for the later
sections. In Section 2, we specialize to study non-semisimple Hopf
algebras of dimension $4p$, and prove Theorem \ref{t:Gorder}. In Section \ref{s:braidedHopf}, we turn our focus to $p$-dimensional braided Hopf algebras $R$ in the Yetter-Drinfeld category $\YD{H_4}$. We show that these algebras are always semisimple, and we provide a necessary condition in Theorem \ref{t:1-chi} for a non-commutative $R$. This condition implies $R$ is commutative for $\dim R=3,7, 11$. For $\dim R=5$, the commutativity of $R$ is shown in Section \ref{s:dim 5}. Finally, we prove the technical Lemma \ref{l:skewprimitive}, and complete the proof of Theorem \ref{t:2844} in Section \ref{s:main}.

\section{Notation And Preliminaries}\label{s:1}
Let $H$ be a finite-dimensional Hopf algebra over $\k$ with comultiplication $\Delta$,
counit $\e$ and antipode $S$. A \emph{left integral} of $H$ is an element $\Lambda \in H$ such that $h\Lambda =\e (h)\Lambda$,
and a \emph{right integral} of $H$ can be defined similarly. The subspace of left (or right) integrals of $H$ is always 1-dimensional (cf. \cite{Mont93bk}).

The dual $H^*$ of $H$ admits a natural Hopf algebra structure and the comultiplication $\Delta$ induces two natural actions of $H^*$, $\rightharpoonup$ and $\leftharpoonup$, on $H$ given by
$$
f \rightharpoonup h =h_1 f(h_2), \quad \text{and} \quad h \leftharpoonup f =  f(h_1) h_2 \quad \text{for } h\in H, \, f \in H^*\,,
$$
where $\Delta(h)= h_1 \o h_2$ is the Sweedler notation for comultiplication with the summation suppressed.
One can easily see that an element $\lambda \in H^*$ is a right
integral if, and only if,
\begin{equation}
  h \leftharpoonup \lambda = \lambda(h)1_H\quad \text{for all }h \in H\,.
\end{equation}

A non-zero element  $g \in H$ is said to be  \emph{group-like} if $\Delta(g) = g \otimes g$. The set of all group-like elements of $H$, denoted by $G(H)$, forms a group under the multiplication of $H$, and  is linearly independent over $\k$. The \emph{distinguished group-like} elements $a \in G(H)$ and $\a \in G(H^*)$ are respectively defined by the conditions:
$$
\Lambda h =\a(h) \Lambda, \quad \lambda \rightharpoonup h= \lambda(h) a\quad \text{for all }h \in H,
$$
where $\lambda \in H^*$ is a non-zero right integral, and $\Lambda \in H$ is a non-zero left integral. The fourth power of the antipode can be expressed in terms of the distinguished group-like elements $a \in H$  and $\a \in H^*$ by the celebrated  Radford formula \cite{Radf76}:
\begin{equation} \label{eq:radfordeq}
S^4(h)=a(\alpha\rightharpoonup h \leftharpoonup \alpha^{-1})a^{-1}
\hbox{ for }h\in H.
\end{equation}
The antipode conceives  important information of the underlying Hopf algebra. The following theorem by Larson and
Radford (cf. \cite{LaRa87}, \cite{LaRa88}, \cite{Radf94}) manifests how the antipode relates to the semisimplicity of a Hopf algebra.
 \begin{thm}
Let $H$ be a finite-dimensional Hopf algebra over
the field $\k$ with antipode $S$. Then the following statements are equivalent:
\enumeri{
  \item $H$ is  not semisimple;
  \item $H^*$ is not  semisimple;
  \item $S^2 \ne \id_H$;
  \item $\Tr(S^2)= 0$;
  \item $\Tr(S^2\circ r(b))=0$ for all $b\in H$,
}
where $r(b)$ denotes the linear operator on $H$ defined by $r(b)(x)=xb$.
\end{thm}
This well-known result will be used repeatedly without further explanation in the sequel.

Recall that for an $H$-module $V$, the left dual $V^{\vee}$ is the
left $H$-module with the underlying space $V^*=\Hom_{\k}(V,\k)$ and
the $H$-action given by
$$
(hf)(x)=f(S(h)x) \hbox{ for }x\in V,\, f\in V^*,\, h \in H.
$$
The right dual $^{\vee}V$ of $V$ is defined similarly but the $H$-action is given by
$$
(hf)(x)=f(S^{-1}(h)x).
$$
The collection $\C{H}$ of all the $H$-modules of finite dimension over $\k$ forms a rigid monoidal category.

For any $U, V, W \in \C{H}$, we have the following natural isomorphisms:
\begin{equation}\label{eq:iso}
\Hom_H(V\du \o U , W) \cong \Hom_H(U, V \o W)  \cong \Hom_H(U \o \ldu W, V)\,.
\end{equation}

For an algebra automorphism $\sigma$ on $H$, denote by $_{\sigma}V$ the
$H$-module with the underlying space $V$ and the $\sigma$-twisted action given by $
h\cdot v = \sigma(h)v \hbox{ for } h \in H, v \in V$.
The natural isomorphism $j: \ld{S^2}V \to V\bidu$ of vector spaces is an $H$-module map. In particular, we have
\begin{equation} \label{eq:doubledual}
  \ld {S^2} V \stackrel{j}{\cong} V\bidu\quad \text{for } V \in \C{H}\,.
\end{equation}
Similarly, one also has the $H$-module isomorphism $_{S^{-2}}V\stackrel{j}{\cong}
{\bidu V}$.

Let $P(V)$ and $I(V)$ respectively denote the projective cover and the injective hull of $V \in \C{H}$, and $\Irr(H)$ a complete set of non-isomorphic simple $H$-modules. For $\beta \in G(H^*)$, we define $\k_{\beta}$ as the 1-dimensional $H$-module which affords the irreducible character $\beta$. We will simply write $\k$ for the trivial 1-dimensional $H$-module $\k_\e$. If $\a \in H^*$ is the distinguished group-like element and $V \in \Irr(H)$, then it follows from \cite[Lemma 1.1]{Ng07} that
$$
\k_{\alpha^{-1}}\o {\bidu V}\cong \Soc(P(V)) \hbox{ and } V\cong
\Soc(P(\k_{\alpha}\o V\bidu)).
$$
Since $H$ is a Frobenius algebra, we have
\begin{align}\label{eq:socle}
I(\k_{\alpha^{-1}}\o {\bidu V})\cong P(V)\,.
\end{align}
 For  $M, V \in \C{H}$, we define $[M:V]:=\dim\, \Hom_H(P(V), M)$. If $V$ is simple, $[M:V]$ is equal to the multiplicity of
 $V$ appearing as a composition factor of
$M$. By  \cite[Lemma 1.7.7]{Benson}, we have
\begin{equation}\label{eq:mult2}
 [M:V]  = \dim\, \Hom_H(M, I(V)).
\end{equation}
\begin{lem}\label{l:composition}
Let $V, W$ be  simple $H$-modules. Then
$$
[P(V): W]= [P(W): \k_{\a\inv} \o  {}\bidu V]\,.
$$
In addition, if $\dim V \ne \dim W$ and $[P(V):W] > 0$, then
\begin{eqnarray*}
\dim P(W) & \ge & 2  \dim W + [P(V): W]\dim V \quad \text{and}\\
\dim P(V) & \ge & 2 \dim V + [P(V): W]\dim W\,.
\end{eqnarray*}

\end{lem}
\begin{proof}
By \eqref{eq:socle} and \eqref{eq:mult2}, we have
\begin{multline*}
[P(V):W] = \dim \Hom_H(P(W), P(V))\\
=\dim \Hom_H(P(W),
I(\k_{\a\inv}\o {}\bidu V)) =[P(W): \k_{\a\inv} \o  {}\bidu V]\,.
\end{multline*}
If $[P(V):W] > 0$ and $\dim V \ne \dim W$, then $[P(W): \k_{\a\inv}
\o  {}\bidu V]>0$. Thus, besides
the socle and the head, $P(W)$ has the composition factor
$\k_{\a\inv}\o {}\bidu V$ with multiplicity $[P(V):W]$. Therefore,
$$
\dim P(W) \ge 2 \dim W + [P(V):W] \dim V\,.
$$
The second inequality can be obtained by the same argument.
\end{proof}
\begin{lem}\label{l:primitive}
Let $H$ be a finite-dimensional non-semisimple Hopf algebra with
Jacobson radical $J$. For any simple $H$-modules $V,W$, we have
$$
[JP(V)/J^2P(V): W] = \dim \Ext(V,W) = [\Soc(I(W)/W): V].
$$
In addition, if $H^*$ has  no non-trivial skew primitive element, then
each simple $H$-submodule of $JP(\k)/J^2P(\k)$ or $\Soc(I(\k)/\k)$ is not 1-dimensional.
\end{lem}
\begin{proof}
Let us abbreviate $P(V)$ as $P$, and consider the natural exact sequence
$$
0 \to JP \to P \to V \to 0\,.
$$
For any simple $H$-module $W$, we have the associated long exact sequence
$$
0 \to \Hom_H(V, W) \to \Hom_H(P, W) \to \Hom_H(JP, W) \to \Ext(V, W) \to \Ext(P,W)\,.
$$
Since $\Ext(P,W)=0$ and $\Hom_H(V, W) \to \Hom_H(P, W)$ is an isomorphism, we find
$$
\Hom_H(JP/J^2P, W)\cong \Hom_H(JP, W)  \stackrel{\cong}{\to} \Ext(V, W)  \,.
$$
Therefore, $[JP/J^2P: W] = \dim \Ext(V, W)$. The second equality can be obtained similarly by considering
the exact sequence
$$
0 \to W \to I(W) \to I(W)/W \to 0\,.
$$

If $H^*$ does not have any non-trivial skew primitive element, then, by \cite[Lemma 2.3]{EG03}, $\Ext(V, W)=0$ for any 1-dimensional $H$-modules $V,W$. Therefore,
$$
[JP(\k)/J^2P(\k): W] =[\Soc (I(\k)/\k) : W] = 0
$$
for all 1-dimensional $H$-modules $W$.
\end{proof}
The dimensions of projective modules over a finite-dimensional Hopf algebra are of particular importance to the remainder of this paper. The following fact on linear algebra is quite useful to study these dimensions.
\begin{lem}[{\cite[Lemma 1.4]{Ng05}}]\label{l:known}
          Let $V$ be a finite-dimensional vector space over the field $\k$,
          $p$ a prime, and $T$ a linear automorphism on $V$ such that
          $\Tr(T) = 0$. If $T^{p^n} = \id_V$ for some positive
          integer $n$, then $p$ divides the dimension of $V$.
\end{lem}
We close this section with the following corollary which is an application of the preceding lemma.
\begin{cor}\label{cor:proj}
Let $H$ be a non-semisimple Hopf algebra over $\k$ with antipode $S$.
If $\ord(S^2)=p^n$ for some prime $p$, then every indecomposable projective
$H$-module $P$ such that $P \cong P\bidu$ as $H$-modules has dimension divisible by $p$.
\end{cor}
\begin{proof}
Let $P$ be an indecomposable projective $H$-module such that $P \cong P\bidu$. In view of \eqref{eq:doubledual},
$P\cong {_{S^2}P}$. By \cite[Lemma A.2]{HN}, there exists an $H$-module isomorphism $\varphi:
P\rightarrow {_{S^2}P}$ such that $\varphi^{p^n}=\id$. It follows from \cite[Lemma 1.3]{Ng07} that $\Tr(\varphi)=0$.
Therefore, Lemma \ref{l:known} implies that $p$ divides $\dim P$.
\end{proof}

Sweedler's 4-dimensional Hopf algebra $H_4$ is of particular interest in this paper. The Hopf algebra $H_4$ is identical to the Taft algebra $T_2$, which is generated as a $\k$-algebra by
$x,g$ subject to the relations
\begin{equation}\label{eq:Taft}
g^2=1,\quad x^2=0\quad \text{and}\quad gx=-xg\,.
\end{equation}
  The comultiplication $\Delta_{H_4}$, the counit $\e_{H_4}$, and the antipode $S_{H_4}$ are
given by
  $$
  \Delta_{H_4}(x)=x \o g + 1 \o x, \quad \e_{H_4}(x)=0, \quad S_{H_4}(x)=-xg,
  $$
   $$
  \Delta_{H_4}(g)=g \o g , \quad \e_{H_4}(g)=1, \quad S_{H_4}(g)=g\,.
  $$
This notation will be used frequently in Sections \ref{s:braidedHopf} and \ref{s:dim 5}.
\begin{remark}\label{r:H4}
The Hopf algebra $H_4$ is self-dual, and it is the unique non-semisimple Hopf algebra over $\k$ of dimension 4 up to isomorphism.
\end{remark}

\section{Non-semisimple Hopf Algebras with dimension $4p$}
In this section, we will focus on non-semisimple Hopf algebras $H$ of dimension $4p$, where $p$ is an odd prime, and we prove Theorem \ref{t:Gorder}.  In particular,  this completes the classification of non-semisimple Hopf algebras $H$ of dimension $4p$ with $|G(H)|$ or $|G(H^*)| > 2$ by applying the result of Andruskiewitsch and Natale \cite[Lemma A.1]{AN1}. It will be shown in Section \ref{s:main} that this condition on the group-like elements holds for $p \le 11$.

Throughout the remainder of this paper, we assume $p$ is an odd prime, and $H$ is a non-semisimple Hopf algebra of dimension $4p$. We begin with
\begin{lem}\label{l:2p}\mbox{} \enumeri{
\item If $H$
contains a Hopf subalgebra $K$ of dimension $2p$, then $H$ is
pointed and $K$ is isomorphic to a cyclic group algebra.
\item If $H$
contains a Hopf ideal $I$ of dimension $2p$, then $H^*$ is
pointed and $H/I$  is isomorphic to a cyclic group algebra.
}
\end{lem}
\begin{proof}
(i) By \cite{Ng05}, $K$ is semisimple and
so $S_K^2 =\id_K$.  Therefore, there are at least $2p$ eigenvalues of $S^2$,
including multiplicities, which are 1. Since $H$ is not
semisimple, $\Tr(S^2)=0$ and this forces the remaining $2p$
eigenvalues of $S^2$ to be all -1. Thus, $S^4=\id_H$ and the order
of $S^2$ is 2. By \cite[Proposition 5.1]{AnSc98}, $K$ is equal to
the coradical of $H$. It follows from \cite[Lemmas A.1 and A.2]{AN1}
that $H$ is pointed, $K=\k[G(H)]$ and $G(H)$ is a cyclic group of
order $2p$. \medskip\\
(ii) If $I$ is a Hopf ideal of dimension $2p$, then $H/I$ is a Hopf algebra of dimension $2p$ and
$H^*$ admits a Hopf subalgebra isomorphic to $(H/I)^*$. By (i), $H^*$ is pointed and $H/I$ is isomorphic to a cyclic group algebra.
\end{proof}
We precede our discussion with the following remark.
\begin{remark}\label{r:subhopf}
Let $K$ be a commutative semisimple Hopf subalgebra of $H$. By
Nichols-Zoeller Theorem, $H$ is a right free $K$-module of rank $\dim H /\dim K$.
  If $E$ is the set of orthogonal primitive idempotents of $K$, then $H=\bigoplus\limits_{e \in E} He$ is a
decomposition of left $H$-module and $\dim He = \dim H/\dim K$ for all $e \in E$.
\end{remark}

\begin{lem}\label{l:p}
If $p \mid |G(H)|$, then $H$  is pointed.
\end{lem}
\begin{proof}
Let $g\in G(H)$ such that $\ord(g)=p$. First we show that the
composition factors of $P(\k)$ are 1-dimensional.

Let $\{e_1, \dots, e_p\}$ be the complete set of orthogonal primitive
idempotents of $\k[g]$. By the preceding remark,  $H=\bigoplus_i
He_i$ is a decomposition of left $H$-modules, and $\dim H e_i = 4$. By the Krull-Schmidt
Theorem, $\dim P(V) \le 4$ for all $V \in \Irr(H)$. Since $H$
is not semisimple, $\k$ is not projective. Therefore, $2 \le \dim
P(\k)\le 4$.

If $\dim P(\k)=2,3$, then all its composition factors are
1-dimensional. Suppose $\dim P(\k)=4$ and it has a composition factor
$V$ of dimension greater than 1. Then $\dim V =2$, $[P(\k):V]=1$ and
$V$ is not projective. By Lemma \ref{l:composition},
$$
\dim P(V) \ge 2 \dim V +1 > 4, \quad \text{a contradiction!}
$$

Since $P(\k)$ is not simple and all its composition factors  are
1-dimensional, it follows from Lemma  \ref{l:primitive} that $H^*$ admits a non-trivial skew primitive element. Now, we
apply \cite[Proposition 1.8]{AN1} to conclude that $H^*$ contains a
non-semisimple pointed Hopf subalgebra $K$ of dimension $M^2N$
with $M >1$. By Nichols-Zoeller Theorem, $M^2N\big|4p$ and this implies $M=2$
and $N=1$ or $p$.

If $N=p$, then $K=H^*$. By the classification \cite[Lemma A.1]{AN1} of pointed Hopf algebras of dimension $4p$,  either $K^*$ is pointed or $|G(K^*)|=2$. Since $p \mid |G(H)|$, $K^*$ is pointed and  so is $H$.

If $N=1$, then $K$ is isomorphic to the Sweedler algebra $H_4$. By dualizing the inclusion map $K\rightarrow H^*$, we
find a Hopf algebra surjection $\pi: H \to H_4$. Moreover, the
coinvariant $R=H^{\mathrm{co}\,\pi}=\{r \in H\mid r_1 \o \pi(r_2) =
r \o 1\}$ is a (left) normal left coideal subalgebra of dimension $p$ (cf. \cite[3.4.2 (2)]{Mont93bk}).
Since $\pi(g)$ is a group-like element of $H_4$,  $\pi(g)=1$ and hence
$R=\k[g]$. In particular, $R$ is a normal Hopf subalgebra of $H$ by \cite[4.5 (a)]{Tk94}. Therefore, we have the exact sequence of Hopf algebras:
$$
1 \to \k[g] \to H \to H_4\to 1\,.
$$
One can dualize the sequence and apply \cite[Proposition 1.6]{HN} to
conclude that $H$ contains a semisimple Hopf subalgebra of dimension
$2p$. Then, by Lemma \ref{l:2p}, $H$ is pointed.
\end{proof}

\begin{cor}\label{c:Sorder}
  Let $S$ be the antipode of $H$. Then $S^{16}=\id_H$. Moreover, for $V \in \Irr(H)$ such that $V \cong V\bidu$, $\dim P(V)$ is even. In particular, $\dim P(\k)$ is an even integer.
\end{cor}
\begin{proof}
  Let $\a \in H^*$ and $a \in H$ be the distinguished group-like elements. If $p$ divides $\ord(\a)\ord(a)$, then $p \mid |G(H)|$ or $p \mid |G(H^*)|$. It follows from Lemma \ref{l:p} that $H$ or $H^*$ is pointed. By \cite[Lemma A.1]{AN1}, $S^4=\id_H$. On the other hand, if $p \nmid \ord(\a)\ord(a)$, then $\ord(\a)$ and $\ord(a)$ are factors of $4$. Thus, the statement follows from  Radford's formula \eqref{eq:radfordeq} of $S^4$.

  If $V \in \Irr(H)$ is isomorphic to $V\bidu$, then $P(V) \cong P(V)\bidu$, and the second statement follows immediately from Corollary \ref{cor:proj}.
\end{proof}
\begin{lem}\label{l:4}
The order of $G(H)$ is not divisible by $4$ and $S^8=\id$.
\end{lem}
\begin{proof}
Suppose  $4 \mid |G(H)|$. Then $|G(H)|=4$ otherwise $|G(H)|=4p$
which implies $H$ is semisimple. The Hopf subalgebra
$K=\k[G(H)]$ is commutative and semisimple. Let
  $\{e_1, e_2, e_3, e_4\}$ be the set of orthogonal primitive idempotents of $K$. In view of Remark \ref{r:subhopf},
 $$
 \dim He_i = p\,.
 $$
 Note that $S^2$ stabilizes $He_i$ and $\Tr(S^2|_{He_i})=0$. It follows from Lemma \ref{l:known}  and Corollary \ref{c:Sorder} that $\dim H e_i$ is even, a contradiction. Therefore, $4 \nmid |G(H)|$.

 If $H$ or $H^*$ is pointed, then the classification of \cite[Lemma A.1]{AN1} implies $S^4=\id$. If neither $H$ nor $H^*$ is pointed, then Lemma \ref{l:p} forces $p\nmid |G(H)||G(H^*)|$. Thus, $\lcm(|G(H)|,|G(H^*)|) \le 2$ and the second part of the statement follows immediately from Radford's formula of $S^4$.
\end{proof}
Now we can prove Theorem \ref{t:Gorder}.
\begin{proof}[Proof of Theorem \ref{t:Gorder}]
  If $|G(H)|>2$, then it follows from Lemma \ref{l:4}  that  $p \mid |G(H)|$. The result then follows immediately from Lemma \ref{l:p}. Conversely, if $H$ is a pointed Hopf algebra of $4p$, then  $G(H)$ is a cyclic group of order $2p$ by \cite[Lemma A.1]{AN1}. This completes the proof of Theorem I.
\end{proof}

We close this section with an observation on 2-dimensional simple $H$-modules.
\begin{lem}\label{l:dual}
Every 2-dimensional simple $H$-module is  not self-dual.
\end{lem}
\begin{proof}
Suppose $V$ is a simple $H$-module of dimension 2 such that $V\cong V\du$. Then $H^*$ is not pointed otherwise all simple $H$-modules are of dimension 1. Moreover, we have $\ann V = \ann V\du =
S(\ann V)$. Therefore, $S$ induces an isomorphism $\ol S$ on $H/\ann
V$. Since $V$ is simple, $H/\ann V$ is a simple algebra of dimension
4. Let $\pi: H \to H/\ann V$ be the natural surjection. Then $\pi^*:
(H/\ann V)^* \to H^*$ is an injective coalgebra map which satisfies
the commutative diagram:
$$
\xymatrix{
(H/\ann V)^* \ar[r]^-{\pi^*} \ar[d]_-{{\ol S}^*} & H^* \ar[d]^-{S^*}\\
(H/\ann V)^* \ar[r]^-{\pi^*} & H^*  \,. }
$$
In particular, $C=\pi^*((H/\ann V)^*)$ is a simple subcoalgebra of
$H^*$ stabilized by $S^*$.

        Let $K$ be the  subalgebra of $H^*$ generated by $C$. Then $K$ is a Hopf subalgebra of $H^*$ and hence $\dim K=4,p,2p$ or $4p$.
        Since $K$ contains the simple subcoalgebra $C$, $\dim K \neq 4, p$.
        If $\dim K=2p$, then $K$ is a cyclic group algebra by Lemma \ref{l:2p} but this is absurd. So we have $\dim K=4p$, and hence $H^*=K$. By \cite{Na02},
        there exists an exact sequence of Hopf algebras
        $$1\rightarrow \k^G\rightarrow H^*\rightarrow B\rightarrow1,$$
        where $B^*$ is a non-semisimple pointed Hopf algebra and $G$ is a
        finite group. Note that $\dim B$ can only be $4$ or $4p$ since Hopf algebras of dimension $2,2p, p$ are semisimple. If $\dim B=4$,  then $G$ is a cyclic group of order $p$ and it follows from Remark \ref{r:H4} that $B \cong H_4$.  Since $\k^G \cong \k[G]$ as Hopf algebras, $G(H^*)$ contains a subgroup of order $p$. It follows from Lemma \ref{l:p} that  $H^*$ is pointed but this is absurd. This forces $\dim B=4p$ and hence $H^*=B$. In particular, $H$ is pointed. By \cite[Lemma A.1]{AN1}, $H \cong \AA(\tau,1)$. It follows from \cite{Radf75} that every 2-dimensional simple $H$-module is not self-dual, a contradiction.
\end{proof}

\section{$p$-dimensional Hopf algebras in a Yetter-Drinfeld Category}\label{s:braidedHopf}
Let $p$ be an odd prime.
The Radford biproduct or bosonization  $R\times H_4$ of
a $p$-dimensional braided Hopf algebra $R$ in the Yetter-Drinfeld category $\YD{H_4}$ is a non-semisimple Hopf algebra of dimension $4p$. In addition, if $1_R$ is the unique group-like element of $R$, then  $G(R \times H_4)$ is of order 2, and hence, by Theorem \ref{t:Gorder},  $R \times H_4$  is not pointed. However, it remains unclear whether such a braided Hopf algebra exists.

In this section, we mainly study $p$-dimensional braided Hopf algebras in $\YD{H_4}$, and we show in Theorem \ref{t:ss} that  they must be semisimple as a $\k$-algebra. Moreover, the bosonization  $R \times H_4$ is pointed if, and only if, $R$ is commutative. If $R$ is not commutative, we show in Theorem \ref{t:1-chi} that $|G(R^*)|\equiv 1 \mod 4$. This theorem implies that all the 7 or 11 dimensional braided Hopf algebras in $\YD{H_4}$ are trivial Hopf algebras.

Recall that a Yetter-Drinfeld module over a Hopf algebra $B$ is a $\k$-vector space $V$ equipped with  a left $B$-module action $\triangleright$ and a left
$B$-comodule action $\rho_V : V \to B \o V$  which satisfy the compatibility condition:
$$
\rho_V(b\triangleright v) = b_1 v_{-1} S_B(b_3) \o b_2 \triangleright v_{0}
$$
for $b \in B$ and $v \in V$, where $\rho_V(v) = v_{-1}\o v_{0}$ is the Sweedler notation with the summation suppressed again. The Yetter-Drinfeld category $\YD{B}$, which consists of the  Yetter-Drinfeld modules over $B$ as objects, is a braided monoidal category. One can define a Hopf algebra in such category (cf. \cite{Mont93bk} or \cite{Andrus02}), and it is often called a \emph{braided Hopf algebra} in $\YD{B}$. The Radford biproduct or bosonization $R \times B$ constructed from a braided Hopf algebra $R$ in $\YD{B}$ is an ordinary Hopf algebra \cite{Radf85}.
If both $R$ and $B$ are finite-dimensional, then $R^*$ is naturally a braided Hopf algebra in $\YD{B^*}$ and we have an isomorphism
\begin{equation}\label{eq:dualbiprod}
  (R \times B)^* \cong R^* \times B^*
\end{equation}
as Hopf algebras.

The underlying space of $R \times B$ is $R \otimes B$.
For simplicity, we use the identifications $r = r \otimes 1_B$ and $b=1_R \otimes b$ in $R \times B$ for $r \in R$ and $b \in B$. Under this convention, the multiplication, the comultiplication, the counit and the antipode of $R \times B$ are given by
$$
(ra)(sb) := r(a_1 \triangleright s) a_2b , \quad \Delta(rb) := r^{(1)} (r^{(2)})_{-1}b_1 \o (r^{(2)})_{0} b_2,
$$
$$
\e(rb):=\e_R(r)\e_B(b), \quad \text{and}\quad S(rb)=(S_B(b_2)\triangleright S_R(r))S_B(b_1)
$$
for $r, s \in R$, $a, b \in B$ where $\Delta_R(r)=r^{(1)} \o r^{(2)}$ denotes the Sweedler notation for the comultiplication $\Delta_R$ of $R$.

We mainly  focus  on $p$-dimensional braided Hopf algebras in $\YD{H_4}$ in this section, and  they are always semisimple and cosemisimple.
\begin{thm}\label{t:ss}  For any $p$-dimensional braided Hopf algebra $R$ in $\YD{H_4}$,  $R$ and $R^*$ are semisimple algebras.
\end{thm}
\begin{proof}
Let $S$ be the antipode of  $R\times H_4$. Then $R$ is stable under $S^2$. By \cite[Theorem 7.3]{AnSc98}, it suffices to show that $\Tr(S^2|_R)\ne 0$. Suppose $\Tr(S^2|_R)= 0$. By Corollary \ref{c:Sorder}, $\ord(S^2|_R)$ is a positive power of 2. In view of Lemma \ref{l:known}, $\dim R$ is even which is a contradiction.
\end{proof}

For a Hopf algebra $H$ of dimension $4p$ such that $H$ and $H^*$ are not pointed, one would like to know when it can be isomorphic to a bosonization of the form $R \times H_4$. The following proposition provides a sufficient condition.
\begin{prop}\label{p:bosonization}
Let $H$ be a $4p$-dimensional Hopf algebra such that neither $H$ nor $H^*$ is pointed. Then $H \cong R \times H_4$ for some braided Hopf algebra $R$ in $\YD{H_4}$ if, and only if, both $H$ and $H^*$ admit a non-trivial skew primitive element.
\end{prop}
\begin{proof} Since  $R \times H_4$ and $(R \times H_4)^*$ contain a Hopf subalgebra isomorphic to $H_4$, both of them contain a non-trivial skew primitive element. Conversely, we assume both $H$ and $H^*$ admit a non-trivial skew primitive element. It follows from Theorem \ref{t:Gorder} that $|G(H)|=|G(H^*)|=2$. By \cite[Proposition 1.8]{AN1}, there exists a non-semisimple Hopf subalgebra $K$ of $H$ such that $\dim K = mn^2$ with $n > 1$ and $|G(K)|=mn$. Since $|G(H)|=2$, $n=2$ and $m=1$. Therefore, $K$ is isomorphic to the Sweedler algebra $H_4$ (cf.  Remark \ref{r:H4}). By the same reason, $H^*$ also contains a Hopf subalgebra isomorphic to $H_4$. Since $H_4 \cong H_4^*$, there exists a Hopf algebra surjection $\pi: H \to H_4$.
It is well known that
$R:=H^{\co \pi}$ and  $K':=K^{\co \pi}$ are left coideal subalgebras of $H$. Moreover,
$K' \subset R$ and $R$ is a Hopf module in $_{K'}^H\mathcal{M}$. Note that $\pi(K)$ is a Hopf subalgebra of $H_4$, and $\dim K'=1,2$ or $4$. More precisely, $K' \cong \k, \k[X]/[X^2]$ or $H_4$ as $\k$-algebras,
 and they are Frobenius algebras. By \cite{Mas92}, $R$ is a free module over $K'$. This forces $\dim K'=1$ as $\dim R =p$. Thus, $\pi(K)=H_4$ or $\pi|_K: K \to H_4$ is a Hopf algebra isomorphism. It follows from \cite{Radf85} that $R$ is a braided Hopf algebra in $\YD{H_4}$, and $H \cong R \times H_4$ as Hopf algebras.
\end{proof}

To proceed with our discussion, we further establish some basic facts on semisimple braided Hopf algebras in a Yetter-Drinfeld category. The readers are referred to \cite{FMS97} for more details.

Let us assume that both $B$ and $R$ are finite-dimensional. By \cite[Remark 5.9]{FMS97}, $R$  is a
Frobenius algebra with the Frobenius map $\l_R$ given by a
right integral of $R^*$, i.e. $\l_R(r^{(1)}) r^{(2)} =\l_R(r) 1_R$ for all $r
\in R$. Moreover, $\l_R$ determines an element $\chi \in G(B^*)$ by the equation
\begin{equation}\label{eq:chi1}
\l_R(b\triangleright r) = \chi(b)\l_R(r) \quad \text{for all }b\in
B,\, r \in R\,.
\end{equation}
The algebra $R$ also admits a right integral $\Lam_R \in R$ such that
$\l_R(\Lam_R)=1$. This implies that
\begin{equation}\label{eq:chi2}
b\triangleright \Lam_R = \chi(b)\Lam_R \quad \text{for all }b\in B.
\end{equation}
If $R$ is semisimple, it has been proved in \cite[p4885]{FMS97} that $\e_R(\Lam_R)\ne 0$ and $\chi=\e_B$. This implies
\begin{equation}\label{eq:chi3}
 \l_R(b\triangleright r) = \e_B(b)\l_R(r) \quad\text{and}\quad b\triangleright \Lam_R = \e_B(b)\Lam_R
\end{equation}
for $b \in B$. We also need the following lemma for semisimple braided Hopf algebras.
\begin{lem} \label{l:chi1} Let $B$ be a finite-dimensional Hopf algebra over $\k$, and
   $R$ a braided Hopf algebra in the Yetter-Drinfeld category $\YD{B}$ such that $R$ is a semisimple $\k$-algebra. Suppose $\Lambda_R$ is a right integral of $R$ and $\l_R$ is a right integral of $R^*$.
   Then
  \begin{equation}\label{eq:one}
    \rho_R(\Lambda_R) = 1_B \o \Lambda_R,  \quad S_R(\Lam_R)=\Lam_R   \quad \text{and}\quad r_{-1}\l_R(r_{0})=1_B \l_R(r)
  \end{equation}
  for $r \in R$.
\end{lem}
\begin{proof}
From \cite[1.6]{Doi00}, there exists a group-like element $z \in B$ such that
$$
\rho_R(\Lambda_R)=S_B(z)\o \Lambda_R\quad \text{and} \quad \l_R(r) z = S_B\inv(r_{-1}) \l_R(r_0)
$$
for $r\in R$. Since $\e_R(\Lambda_R)\ne 0$ and $\e_R$ is a $B$-comodule map, i.e.
$\e_R(\Lam_R)1_B=S_B(z)\e_R(\Lam_R)$, we find $S_B(z)=1_B$. Hence,  the first and the last equalities follow immediately.

Note that $S_R$ is a bijection which satisfies $S_R(rs)=S_R(r_{-1}\triangleright s)S_R(r_0)$ for all $r,s \in R$. In particular, we have
$$
\e_R(s) S_R(\Lambda_R)= S_R(\Lambda_R s)=S_R(s)S_R(\Lambda_R) \,.
$$
Since $\e_R \circ S_R =\e_R$, $S_R(\Lam_R)$ is also a left integral. The semisimplicity of $R$ implies $S_R(\Lam_R)=\g \Lam_R$ for some $\g\in \k$.  Since $\e_R(\Lam_R)\ne 0$, we find $\g=1$.
\end{proof}

Now, we return to our specific case where $B$ is the Sweedler algebra $H_4$. The elements $x, g \in H_4$ defined in the end of Section \ref{s:1} will be used in the reminder of Section \ref{s:braidedHopf} and Section \ref{s:dim 5}.

Let us first consider $H_4$-module algebras which are semisimple as $\k$-algebras.
\begin{lem} \label{l:e1}
Let $A$ be a finite-dimensional left (resp. right) $H_4$-module algebra.
If $A$ is a semisimple algebra and $e$ is a central idempotent of $A$  such that $I=Ae$ is closed under the action of
$g$, then $I$ is a $H_4$-submodule of $A$, and
$$
 g \triangleright e =e\quad \text{and} \quad x\triangleright e
=0.
$$
In particular, $I$ is also a $H_4$-module algebra.
\end{lem}
\begin{proof}
If $e$ is not a primitive idempotent, write $e=e_1+\cdots+e_k$ as a sum of orthogonal
 primitive central idempotents of $A$. Then
$e_ie=e_i=ee_i\in I$ for $1\leq i\leq k$. Note that $g$ acts on $A$ as an algebra automorphism.
Since $I$ is an ideal
closed under $g$-action, $g \triangleright e_i\in I$ is a primitive
central idempotent of $A$, and the action of $g$ permutes the idempotents
$e_1,\dots, e_k$. Thus, we have $g \triangleright e = e$. Note that
\begin{equation}\label{eq:xaction}
x \triangleright e=x\triangleright
e^2=(x\triangleright e)  (g \triangleright e) + e (x \triangleright
e)=2e(x \triangleright e) \in I.
\end{equation}
The equality also implies that $x \triangleright e = 2 x \triangleright e$. Since $\k$ is of characteristic zero, $x \triangleright e=0$. Now,
for any $r\in R$, we find
$x\triangleright (re)=(x \triangleright r)e  \in I$. Therefore, $I$ is closed under the  $H_4$-action.
\end{proof}

\begin{lem} \label{l:e2}
Let $R$ be a finite-dimensional braided Hopf algebra in $\YD{H_4}$. If  $R$ is a semisimple algebra and $I$ is a 1-dimensional ideal of $R$, then
$x\triangleright I =0$.
\end{lem}
\begin{proof} Since $R$ is semisimple, there exists a primitive central idempotent $e_1$ of $R$
such that $I=Re_1$.
If $g \triangleright e_1 = e_1$, then $x \triangleright e_1=0$ by
Lemma \ref{l:e1}. Now, we assume $g\triangleright e_1=e_2 \ne e_1$. Then $e_2$ is also a primitive
central idempotent of $R$ and $e_1 e_2=0$.
The ideal
$\hat I=Re_1+Re_2=R(e_1+e_2)$ is closed under the action of $g$. By Lemma \ref{l:e1},
$x\triangleright
e_i=\alpha_{i1}e_1+\alpha_{i2}e_2$ for some $\alpha_{ij} \in \k$ and
$i,j\in\{1,2\}$, and $x\triangleright (e_1+e_2)=0$. These equations imply
$\alpha_{11}=-\alpha_{21}$ and $\alpha_{12}=-\alpha_{22}$. By the equation $gx\triangleright
e_1=-x\triangleright e_2$, we find $\alpha_{12}=-\alpha_{21}$ and
$\alpha_{11}=-\alpha_{22}$. Thus,  we have
$x\triangleright e_1=\alpha(e_1+e_2)$ and $x\triangleright
e_2=-\alpha(e_1+e_2)$, where $\alpha=\alpha_{11}$. Applying \eqref{eq:chi3} with $b=g$ and
$r=e_1$,  we obtain $\lambda_R(e_1)=\lambda_R(e_2)$. Since
$\ker\lambda_R$ cannot have any non-zero ideal of $R$, $\l_R(e_1)\ne 0$. Applying
\eqref{eq:chi3} again with $b=x$ and $r=e_1$, we find $2\alpha\lambda(e_1)=0$ which
implies $\alpha=0$. Therefore, $x\triangleright \hat I=0$.
\end{proof}

\begin{prop}\label{p:commu-braided}
Let $R$ be a $p$-dimensional Hopf algebra in $\YD{H_4}$. The following statements are equivalent:
\enumeri{
\item $R$ is commutative;
\item $R$ is cocommutative;
\item $R \times H_4$ is pointed;
\item $(R \times H_4)^*$ is pointed;
\item $R$ admits a $g$-invariant 1-dimensional ideal $I$ which is not spanned by any non-zero integral of $R$.
\item The action of $x$ on $R$ is trivial.
}
If one of the equivalent statements holds, then the $H_4$-action and the $H_4$-coaction of $R$ are trivial. In particular, $R$ is an ordinary Hopf algebra.
\end{prop}
\begin{proof} Let $H$ be the biproduct $R \times H_4$. By \eqref{eq:dualbiprod}, $H^* \cong R^* \times H_4^*$.
Note that $R$ and $R^*$ are semisimple algebras by Theorem \ref{t:ss}.\\
(i) $\Rightarrow$ (iv) and (vi). Since $R$ is a commutative semisimple
algebra, $R$ is a direct sum of 1-dimensional ideals.
By Lemma~\ref{l:e2},
$x\triangleright R=0$.  Then $Hx$ is a Hopf ideal of $H$ and $\dim Hx = 2p$. Therefore, $H^*$ is pointed by Lemma \ref{l:2p}. \\
(ii) $\Rightarrow$ (iii). Since $R$ is cocommutative, $R^*$ is a commutative braided Hopf algebra in $\YD{H_4^*}$. Since $H_4^* \cong H_4$, the isomorphism in \eqref{eq:dualbiprod} and the preceding paragraph imply $R \times H_4$ is pointed.\\
(iii) $\Rightarrow$ (i) and (ii). If $H$ is pointed, then $G(H)$ is a cyclic group of order $2p$. Since $|G(H)|>2$, it follows from \cite[2.11]{Radf85} that there exists a non-trivial group-like element $r \in G(R)$ such that $\rho_R(r) = 1_{H_4} \o r$. Then $r$ is a group-like element of $H$, $\ord(r) = p$ and $R=\k[r]$. In particular, $R$ is an ordinary Hopf algebra and the $H_4$-coaction on $R$ is trivial.\\
(iv) $\Rightarrow$ (i), (ii), (v) and (vi). By \eqref{eq:dualbiprod}, $R^* \times H_4^*$ is pointed.
Using the same argument as in the preceding paragraph,  $R^*$ is isomorphic to the ordinary Hopf algebra $\k[\BZ_p]$, and the $H_4^*$-coaction on $R^*$ is trivial. Therefore, the $H_4$-action on $R$ is trivial. In particular, $g$ acts as identity on $R$, and $x$ acts trivially.\\
(v) $\Rightarrow$ (iv). Let $\chi$ be the character of $R$ associated with the $1$-dimensional ideal $I$. Then $\chi\in G(R^*)$ and $\chi\ne \e_R$. Since $I$ is $g$-invariant, $x \triangleright I=0$ by Lemma \ref{l:e2}. Thus, $\rho_{R^*}(\chi)= \e_{H_4} \o \chi$. In view of \cite[2.11]{Radf85}, $R^* \times H_4^*$ has at least 3 group-like elements. By Theorem \ref{t:Gorder}, $(R \times H_4)^*$ is pointed.\\
(vi) $\Rightarrow$ (i). Since $x$ acts on $R$ trivially, $I=H x$ is a
$2p$-dimensional Hopf ideal in $H$. By Lemma \ref{l:2p}, $H/I$ is isomorphic to a cyclic group algebra of dimension $2p$. Since $R$ is isomorphic to the subalgebra $R+ I$ of $H/I$, $R$ is commutative.
\end{proof}
We continue to study the possible non-commutative $p$-dimensional braided Hopf algebras $R$ in $\YD{H_4}$. Since $R$ is a semisimple $\k$-algebra, one can always decompose $R$ into a direct product of simple ideals. Since the group-like element $g \in G(H_4)$ acts as a $\k$-algebra automorphism on $R$, the action of $g$ permutes the simple ideals of $R$. In particular, $A_I=I + (g \triangleright I)$ is $g$-stable for any simple ideal $I$ of $R$. In view of Lemma \ref{l:e1}, $A_I$ is stable under the action $H_4$ and is a $H_4$-module algebra. The ideal $A_I$ is called a $H_4$-\emph{simple ideal} of $R$ in the sequel.

A semisimple braided Hopf algebra $R$ in $\YD{H_4}$ can be decomposed into a direct product of $H_4$-simple ideals. Namely,
\begin{equation}\label{eq:decomp1}
R = A_1 \oplus \cdots \oplus A_l
\end{equation}
where $A_i = I + (g \triangleright I)$ for some simple ideal $I$ of $R$. Therefore,
\begin{equation}\label{eq:decomp2}
R\times H_4 = A_1H_4 \oplus \cdots \oplus A_lH_4
\end{equation}
is an ideal decomposition of the biproduct $R \times H_4$.

Let $A$ be a $H_4$-simple ideal of the semisimple braided Hopf algebra $R$ in $\YD{H_4}$. Then $A$ does not admit any proper $H_4$-invariant ideal of $R$. Let $e_0 = (1_{H_4}+g)/2$ and $e_1= (1_{H_4}-g)/2$. Then $e_0, e_1$ are idempotents of $H_4$ and $\{e_0, e_1, xe_0, xe_1\}$ forms a $\k$-basis for $H_4$. Moreover, it forms an $R$-basis for $R \times H_4$.

Suppose
\begin{equation}\label{eq:E}
E=r_1 e_0 + r_2 e_1+ r_3 xe_0+r_4  xe_1
\end{equation}
is a central idempotent of $AH_4$ for some $r_1, r_2, r_3, r_4 \in A$. Since $gE = Eg$, we find
\begin{equation}\label{eq:g-action}
g\triangleright r_1=r_1,\quad g\triangleright r_2=r_2,\quad
g\triangleright r_3=-r_3,\quad g\triangleright r_4=-r_4.
\end{equation}
Similarly, by the equality $xE =Ex$, we find
\begin{equation}\label{eq:xa}
x \triangleright r_1 = x \triangleright r_2=0 \text{ and } x\triangleright r_3 = x\triangleright r_4 = r_1-r_2\,.
\end{equation}
Note that
$$
E e_0= r_1 e_0 +  r_3  xe_0, \quad  E e_1= r_2  e_1+ r_4  x e_1
$$
are idempotents.
This implies
\begin{equation}\label{eq:as}
r_1^2=r_1, \, r_2^2=r_2,\, r_3=r_1 r_3 + r_3 r_2,\,r_4=r_2 r_4 + r_4 r_1\,.
\end{equation}
For $r \in A$, the commutativity $rE=Er$ implies
\begin{equation}\label{eq:c1}
[r_1,  r] = -r_3 (x \triangleright r) , \quad [r_2, r] = r_4 (x \triangleright r) , \quad [r_3, r]=  [r_4, r]=0
\end{equation}
whenever $g\triangleright r=r$, and
\begin{equation}\label{eq:c2}
r r_1=r_2 r  + r_4(x \triangleright r), \quad r r_2=r_1 r -r_3(x \triangleright r),\quad r r_3 = r_4 r,\quad r r_4 = r_3 r
\end{equation}
whenever $g\triangleright r=-r$. In particular, we have
\begin{equation}\label{eq:commutativity}
[r_i, r_j]=0 \text{ for } i\in \{1,2\}, j\in \{3,4\}, \text{ and }
[r_3+r_4, r] = 0 \text{ for all } r \in A\,.
\end{equation}
In view of \eqref{eq:as},
\begin{equation}\label{eq:34}
r_3 = r_3 (r_1+r_2)=  (r_1+r_2) r_3 \text{ and } r_4 = r_4 (r_1+r_2)=  (r_1+r_2) r_4\,.
\end{equation}
\begin{lem}\label{l:ideals}
Let $I$ be a simple ideal of $R$ of dimension $n^2$.
\enumeri{
\item If $A_I$ is a simple ideal of $R$ or the action of $x$ on $A_I$ is trivial, then $A_IH_4$ is an indecomposable ideal of $R \times H_4$.
\item If $A_I=I \oplus (g\triangleright I)$, then
every simple quotient algebra of $A_I H_4$ has dimension  $4n^2$. Moreover,  $A_I  H_4$ is either a direct sum of two simple ideals of dimension $4n^2$ or a primary algebra, i.e $\frac{A_I  H_4}{Rad(A_I  H_4)}$ is a simple algebra of dimension $4n^2$.
\item If $A_I=I \oplus (g\triangleright I)$ and the action of $x$ on $A_I$ is trivial, then $A_IH_4$ is a primary algebra and $\frac{A_I H_4}{Rad(A_I H_4)}$ is a simple algebra of dimension $4n^2$.
}
In particular, $A_IH_4$ is either an indecomposable ideal or a direct sum of two simple ideals of dimension $4n^2$.
\end{lem}
\begin{proof} (i) Suppose $E$, given by \eqref{eq:E}, is a non-zero central idempotent of $A_IH_4$.
If $A_I$ is simple, then \eqref{eq:commutativity} implies that $r_3+r_4=\g 1_{A_I}$ for some scalar $\g \in \k$. By Lemma \ref{l:e1}, $x \triangleright (r_3 + r_4) =0$. Hence, by \eqref{eq:xa}, $r_1=r_2$. If the action of $x$ on $A_I$ is trivial, then we also have $r_1=r_2$ by \eqref{eq:xa}.
    It follows from \eqref{eq:as} that $r_1$ is a non-zero idempotent, otherwise $E=0$. Thus,  $r_i  = 2r_1 r_i$ for $i=3,4$. This implies that $r_3=r_4=0$ and $E=r_1$. In particular,  $r_1$ is a non-zero central idempotent of $A_I$. Since  $A_I$ is $H_4$-simple and $r_1$ is $g$-invariant by \eqref{eq:g-action},  $r_1$ must be the identity of $A_I$. Thus, there is no non-trivial central idempotent in  $A_I \times H_4$.\\\\
(ii) For the first statement, it is equivalent to show that every simple $A_IH_4$-module $V$ has dimension $2n$. By restriction, $V$ is an $A_I$-module and so $V\cong m_1 V_1 +m_2 V_2$ as $A_I$-modules where $V_1, V_2$ are the simple modules of $I$ and $g \triangleright I$ respectively. We claim that $m_1=m_2=1$. Let $z_1, \dots, z_{m_1}\in V$ such that
    $Iz_i$ is a simple $I$-module and $IV=Iz_1 \oplus\cdots \oplus Iz_{m_1}$. Then
    $gIV=gIz_1 \oplus\cdots \oplus gIz_{m_1}$ is a decomposition of vector subspaces and
    $gIz_i = (g \triangleright I) gz_i$ is $g \triangleright I$-module. Therefore, $m_1 \le m_2$. Similarly, one can show that $m_2 \le m_1$, and hence $m_1 =m_2 = m$. Therefore, $\dim V= 2mn$, and so $\frac{A_I H_4}{Rad (A_I  H_4)}$ contains a simple ideal of dimension $4m^2n^2$. This forces $m=1$ and $\dim V = 2n$.

    If $A_I  H_4$ is not a primary algebra, then there is more than one simple module over $A_I  H_4$. Thus, we have the inequality
    $$
    8n^2 =\dim A_IH_4 \ge \dim \left(\frac{A_I  H_4}{Rad (A_I  H_4)}\right) \ge 8n^2.
    $$
    Therefore, $Rad (A_I  H_4)=0$ and $A_I  H_4$ is a direct sum of two simple ideals of dimension $4n^2$.\\\\
    (iii) If the action of $x$ on $A_I$ is trivial, then $A_I H_4x$ is a nilpotent ideal of $A_I H_4$. In particular, $A_I H_4$ is not semisimple. In view of (ii), $A_I  H_4$ is a primary algebra.
 \end{proof}
\begin{thm} \label{t:1-chi}
Let $p$ be an odd prime and $R$  a non-commutative $p$-dimensional braided Hopf algebra in $\YD{H_4}$. Then $|G(R^*)|$ and $|G(R)|$ are integers congruent to 1 modulo 4.
\end{thm}
\begin{proof} By Proposition \ref{p:commu-braided}, $R^*$ is also non-commutative. It suffices to show that $|G(R^*)| \equiv 1 \mod 4$.
  Since $g$ acts on $R$ as an algebra automorphism, $g$ permutes the simple ideals of $R$. In particular, $g$ permutes the set $\OO$ of $1$-dimensional ideals of $R$ and $g$ fixes the ideal spanned by a non-zero integral $\Lam_R$ of $R$. If $g$ fixes another 1-dimensional ideal which is not spanned by $\Lam_R$, then  it follows from Proposition \ref{p:commu-braided} that $R$ is commutative. Therefore, $g$  fixes exactly one element of $\OO$ and hence $|\OO|$ is an odd number.

  Now, we assume that $|\OO|\equiv 3 \mod 4$. Let $I$ be a 1-dimensional ideal of $R$ such that $\Lam_R \not\in I$ . Then $A =I \oplus (g \triangleright I)$ is a 2-dimensional $H_4$-simple ideal. In view of Lemma \ref{l:e2}, the action of $x$ on $A$ is trivial. It follows from Lemma \ref{l:ideals} (iii) that $AH_4$ is an indecomposable ideal of $R \times H_4$ and $\dim AH_4=8$. By Lemma \ref{l:ideals}, for any simple ideal $J$ of $R$ with $\dim J > 1$, each indecomposable ideal summand of $(J+(g \triangleright J)) H_4$ has dimension at least  16. Therefore,  $\widehat{\OO}=\{A_I H_4\mid I \in \OO,\, I\ne \k \Lam_R\}$ are all the 8-dimensional indecomposable ideal summands of $R \times H_4$. Moreover, each of these indecomposable summands is primary and its simple quotient is 4-dimensional.  The antipode $S$ of $R \times H_4$ permutes the ideals in $\widehat\OO$. Since $|\widehat\OO|$ is an odd number and  $\ord(S)$ is a power of $2$ by Lemma \ref{l:4}, one of these $8$-dimensional ideals is stable under $S$. Therefore, the associated 2-dimensional simple $(R \times H_4)$-module is self-dual. However, this contradicts Lemma \ref{l:dual}. Therefore $|\OO|\equiv 1 \mod 4$.

    Since $|G(R^*)|=|\OO|$, the result follows.
\end{proof}
\begin{cor}\label{c:347}
  Braided Hopf algebras of dimension 3, 7 or 11 in $\YD{H_4}$ are trivial Yetter-Drinfeld modules, and they are group algebras.
\end{cor}
\begin{proof} Note that these braided Hopf algebras $R$ are semisimple by Theorem \ref{t:ss}. Thus,  $R$ is obviously commutative when $\dim R=3$. Suppose there exists a braided Hopf algebra $R$ in  $\YD{H_4}$ of dimension 7 or 11 which is not a commutative algebra. The number of 1-dimensional characters of $R$ is either congruent to 3 modulo 4 or equal to 2, but this contradicts Theorem \ref{t:1-chi}. Therefore, $R$ must be commutative, and  the result follows immediately from Proposition \ref{p:commu-braided}.
\end{proof}

\section{Braided Hopf algebras of dimension 5 in $\YD{H_4}$}\label{s:dim 5}
  Theorem \ref{t:1-chi} cannot be applied to decide the commutativity of a braided Hopf algebra of dimension congruent to 1 modulo 4. In this section, we use the result of Chen and Zhang  \cite{CZ} on the classification of 4-dimensional matrix algebra in $\CM{D(H_4)}$ to deal with the 5-dimensional case, and  we prove
\begin{lem}\label{l:b5}
Braided Hopf algebras of dimension 5 in $\YD{H_4}$ are commutative.
\end{lem}
 The proof of Lemma \ref{l:b5} will be elaborated in the remainder of this section. We will continue to use the convention for the Sweedler algebra $H_4$  introduced in the end of Section \ref{s:1}.

\begin{remark}\label{r:1}
The classification of 4-dimensional matrix algebras in $\YD{H_4}$ follows immediately from \cite[Theorem 4.10]{CZ}. In general, for any finite-dimensional Hopf algebra $B$, the categories $\CC=\CM{D(B)}$ and $\CD=\YD{B}$ are isomorphic as monoidal categories under the strict monoidal functor $\CF: \CC \to \CD$ defined as follows:
for $M \in \CC$, the $B$-module structure of $\CF(M)$ is the $B$-restriction of $M$, and the left $B$-comodule structure $\rho: M \to B \o M$ of $\CF(M)$ given by
\begin{equation}\label{eq:comultformula}
 \rho(v) = \sum_{i} S_B(b_i) \o b^i v
\end{equation}
for $v \in M$, where $\{b_i\}_i$ is a basis for $B$ and $\{b^i\}_i$ is its dual basis. For any morphism $f: M \to N$ in $\CC$,  the assignment $\CF(f):=f$ is a morphism from $\CF (M)$ to $\CF(N)$ in $\CD$.
Since $M=\CF(M)$ as vector spaces, $\dim M =\dim \CF(M)$ for all $M \in \CC$.

This isomorphism of monoidal categories implies that $A$ is an algebra in $\CC$ if, and only, if $\CF(A)$ is an algebra in $\CD$. In this case, $A =\CF(A)$ as $\k$-algebras. Moreover, $A$ and $B$ are isomorphic algebras in $\CC$ if, and only if,  $\CF(A)$ and $\CF(B)$ are isomorphic algebras in $\CD$.
\end{remark}

 Assume  $R$ is a 5-dimensional non-commutative braided Hopf algebra in $\YD{H_4}$ with left $H_4$-action  $\triangleright$ and left $H_4$-coaction $\rho_R$.
  By Theorem \ref{t:ss}, $R$ is a semisimple $\k$-algebra, and so
 $R$ has a unique ideal decomposition $R=A\oplus I$
where $A$ is a 4-dimensional ideal isomorphic to a matrix algebra
and $I$ is an 1-dimensional ideal generated by  the normalized integral $e=\Lambda_R/\e_R(\Lambda_R)$. Hence, $\e_R(A)=0$ and $e$ is a central idempotent of $R$.  It follows from \eqref{eq:chi3} and Lemma \ref{l:chi1} that
$t\triangleright e = \e_{H_4}(t) e$ for $t \in H_4$, and $\rho_R(e)= 1_{H_4} \o e$. Since the
action of $g$ is an algebra automorphism of $R$, $g \triangleright A
= A$ and hence $g$ fixes the central idempotent $\iota$ of $A$.
By Lemma \ref{l:e1}, $A$ is closed under the $H_4$-action of $R$ and  $x
\triangleright \iota =0$. Thus, $A$ is also a $H_4$-module algebra.

The left $H_4$-comodule algebra structure of $R$ induces a right
$H_4^*$-module algebra on $R$. Since $H_4^* \cong H_4$ as Hopf algebras,
by the preceding arguments, $A$ is also a right $H_4^*$-submodule of $R$
and hence a $H_4$-subcomodule of $R$. Thus, $A$ is a 4-dimensional
matrix algebra in the category $\YD{H_4}$. Since $R$ is not commutative and $x$ acts trivially on $I$, the $x$-action on  $A$ cannot be trivial by Proposition \ref{p:commu-braided}.

In view of Remark \ref{r:1}, we apply the classification of 4-dimensional matrix algebra in $\YD{H_4}$ follows from \cite[Theorem 4.10]{CZ}. Since the $x$-action on $A$ is trivial and $\k$ is algebraically closed, the Yetter-Drinfeld $H_4$-algebra $A$ can only be isomorphic to the corresponding classes (c), (h), (i), (k), (m) or (n) of \cite[Theorem 4.10]{CZ}. There exist $u,v \in A$ such that
 $\{\iota, u, v, uv\}$ forms a basis for $A$ and satisfy the relations
\begin{equation}\label{eq:alg}
u^2 =\alpha \iota, \quad v^2 = \b \iota, \quad uv+vu = \g \iota
\quad \text{for some } \alpha, \b, \g \in \k\,,
\end{equation}
and the $H_4$-action $\triangleright$  is given by
\begin{equation}\label{eq:module}
\begin{aligned}
 & x\triangleright u=0,  \quad\,   x\triangleright v=\iota,
  & g\triangleright u=-u,  \quad g\triangleright v=-v \,.
\end{aligned}
\end{equation}

The $H_4$-comodule structures on $A$ can only be one of the following
three cases:
\begin{equation}\label{eq:cases}
\begin{array}{cl}
 \mbox{(A)} & \rho_R(u)=1_{H_4} \o u + 2 x \o uv, \quad \rho_R(v)=g \o v - 2\b xg \o \iota  \text{ and } \g=0; \medskip \\
 \mbox{(B)} &  \rho_R(u)=g \o u, \quad\rho_R(v)=\eta xg \o \iota + g \o v \text{ where } \eta \in \k; \medskip\\
 \mbox{(C)} & \rho_R(u)=xg\o \iota+g\o u,\quad \rho_R(v)=g\o v. \medskip
\end{array}
%\enumerI{
% \item $\rho(u)=\iota \o u + 2 x \o uv, \quad \rho(v)=\iota \o u - 2\b xg \o \iota  \text{ and } \g=0.$ \medskip
%\item  $\rho(u)=g \o u, \quad\rho(v)=\eta xg \o \iota + g \o v \text{ where } \eta \in \k.$ \medskip
%\item  $\rho(u)=xg\o \iota+g\o u,\quad \rho(v)=g\o v.$ \medskip
%}
\end{equation}
By \eqref{eq:module}, $A$ can be decomposed into a direct sum
$A=P_0 \oplus P_1$ of $H_4$-submodules, where $P_0$ is spanned by $uv, u$
and $P_1$ by $v,\iota$. The $H_4$-modules $P_0$ and $P_1$ are
 two principal modules of $H_4$, i.e. $H_4 \cong P_0 \oplus P_1$.
   Therefore, $R=P_0 \oplus P_1 \oplus I$ is a direct sum
of indecomposable $H_4$-submodules of $R$. Note that
$$
\Hom_{H_4}(P_i, P_j) \cong \Hom_{H_4}(P_0, I) \cong \Hom_{H_4}(I, P_1) \cong
\Hom_{H_4}(I, I) \cong \k
$$
as $\k$-linear spaces for all $i, j=0, 1$. Since the antipode of $R$
is a left $H_4$-module map, we have
$$
\begin{array}{cccccccc}
    S_R(\iota)& = &\zeta_1 \iota & &  & &\\
    S_R(u)& = & & \zeta_2 u  & & & &\\
    S_R(v)& = && \zeta_3 u & + \zeta_1 v& &  \\
    S_R(uv)& =& \zeta_4 \iota & & & +\zeta_2 uv & +\zeta_5 e \\
    S_R(e)& = &\zeta_6 \iota  & & & &+ \zeta_7 e
    \end{array}
$$
    for some $\zeta_1, \dots, \zeta_7 \in \k$. Since $S_R(1_R) =1_R$ and $\e_R \circ S_R =\e_R$, it follows immediately
from Lemma \ref{l:chi1}  that
$$
\zeta_5=\zeta_6=0, \quad\text{and}\quad  \zeta_1=\zeta_7=1.
$$

    Let $\l_R$ be a right integral of $R^*$ such that $\l_R(e)=1$. Since $R^*$ is also a semisimple $\k$-algebra, $\l_R$ is a two-sided integral. In view of \eqref{eq:chi3} and \eqref{eq:module}, $\l_R(\iota)=\l_R(u)=\l_R(v)=0$. Since $\ker \l_R$ cannot contain any non-zero ideal of $R$, $\l_R(uv) \ne 0$. It follows from \eqref{eq:one} of Lemma  \ref{l:chi1} that case (A) of \eqref{eq:cases} is impossible.

    One can renormalize $u$ or $v$ to further assume $\l_R(uv)=1$. Then it follows from \eqref{eq:alg} that
    $\l_R(vu)=-1$. With respect to the Frobenius map $\l_R$,  $\{\iota, u, v, uv, e\}$  and $\{uv, -v, u, \iota, e\}$ form a pair of the dual bases for $R$, i.e.
    $$
    r = \iota \l_R(uv r)-u \l_R(v r)+v \l_R(u r)+uv \l_R(\iota r)+e \l_R(e r) \quad \text{for }r \in R.
    $$
    By \cite{FMS97},
    \begin{multline}\label{eq:dualbasis0}
     \iota \o uv- u \o v +v \o u +uv \o \iota +e\o e \\ = S_R\inv((e^{(2)})_{0}) \o
\e_{H_4}((e^{(2)})_{-1})  S_{H_4}\inv((e^{(2)})_{-2}) \triangleright e^{(1)} \\
    = S_R\inv((e^{(2)})_{0}) \o S_{H_4}\inv((e^{(2)})_{-1}) \triangleright e^{(1)}\,.
    \end{multline}
    Hence
   \begin{multline}\label{eq:dualbasis}
       S_R(\iota) \o uv- S_R(u) \o v +S_R(v) \o u +S_R(uv) \o \iota +S_R(e)\o e \\ = (e^{(2)})_{0} \o
S_{H_4}\inv((e^{(2)})_{-1}) \triangleright e^{(1)}\,.
   \end{multline}
    Applying $\l_R \o \id$ to the equation, we then have
    \begin{multline}
      \zeta_2\iota+e = \l_R((e^{(2)})_{0} )S_{H_4}\inv((e^{(2)})_{-1}) \triangleright e^{(1)} \\ =
\l_R(e^{(2)}) S_{H_4}\inv(1_{H_4}) \triangleright e^{(1)}  = \l_R(e^{(2)})  e^{(1)} = 1_R \l_R(e) =1_R.
    \end{multline}
    Therefore, $\zeta_2=1$.
    Applying the equality
    $$
    S_R\inv((e^{(2)})_{0})  S_{H_4}\inv((e^{(2)})_{-1}) \triangleright e^{(1)} =\e_R(e)1_R
    $$
    (cf. \cite[1.6]{Doi00}) to \eqref{eq:dualbasis0}, we find
    $\g \iota +e=1_R$ and hence $\g=1$. However, this implies $S_R$ cannot be an anti-algebra homomorphism in the braided sense, i.e.
    \begin{equation}\label{eq:briaded-anti-alg}
     S_R(rs) = (r_{-1} \triangleright S_R(s))S_R(r_0)\,.
    \end{equation}
    Suppose $S_R$ satisfies the condition \eqref{eq:briaded-anti-alg}. In both cases (B) and (C),  we find
      $$
    (u_{-1} \triangleright S_R(u))S_R(u_0) =- \a \iota,\quad  S_R(u^2) = \a\iota,
     $$
    $$
    (v_{-1} \triangleright S_R(u))S_R(v_0) =-\zeta_3 \a \iota-uv, \quad S_R(vu)=(1-\zeta_4)
\iota-uv\,.
     $$
     These equalities imply $\a=0$ and $\zeta_4=1$.
     Thus, $S_R(uv) = uv+\iota$. Since
     $$
     (u_{-1} \triangleright S_R(v)) S_R(u_0)=\left\{\begin{array}{cl}
     uv -\iota &\text{for case (B) and}\\
     uv -2\iota &\text{for case (C)},\\
     \end{array}\right.
     $$
     both cases (B) and (C) do not satisfy \eqref{eq:briaded-anti-alg}. Therefore, there does
not exist any non-commutative semisimple braided Hopf algebra of dimension 5 in $\YD{H_4}$.
This completes the proof of Lemma \ref{l:b5}. \qed

\section{Non-Semisimple Hopf algebras of dimension $20$, $28$ and $44$}\label{s:main}
Let $H$ be a non-semisimple Hopf algebra over $\k$ of dimension $4p$ where $p$ is an odd prime.
By Proposition \ref{p:bosonization}, if $H$ and $H^*$ are not pointed and they both admit a non-trivial skew primitive element, then $H$ is a biproduct of the form $R \times H_4$. In this section, we prove the following technical lemma:
\begin{lem}\label{l:skewprimitive}
  If $H$ is a non-semisimple Hopf algebra over $\k$ of dimension $4p$ where $p=3,5,7,11$, then both $H$ and $H^*$ have a non-trivial skew primitive element.
\end{lem}
 As a consequence, we can complete the proof of Theorem \ref{t:2844}.

 \begin{proof}[Proof of Theorem \ref{t:2844}]
  Suppose $H$ and $H^*$ are not pointed. By Lemma \ref{l:skewprimitive} and Proposition \ref{p:bosonization}, $H \cong R \times H_4$ for some braided Hopf algebra $R$ in $\YD{H_4}$ and $\dim R=3,5, 7$ or $11$. Applying Corollary \ref{c:347} and Lemma \ref{l:b5}, $R$ is a commutative algebra. By  Proposition \ref{p:commu-braided}, $R \times H_4$ is a pointed Hopf algebra, a contradiction.
 \end{proof}

The remainder of the section is devoted to prove Lemma \ref{l:skewprimitive}, and we need to establish a few lemmas for the proof. Throughout this section, we assume $H$ is a \emph{non-semisimple} Hopf algebra over $\k$ of dimension $4p$, and let
$$
\Irr_{1}(H)=\{U \in \Irr(H)\mid \dim U \mbox{ is odd }\}.
$$
\begin{lem}\label{l:step1} Let $V$ be a simple $H$-module  with $\dim V > 1$.
Suppose  $V$ is a composition factor of $P(\k)$. Then
   \enumeri{
   \item   $\dim P(V) \ge 2\dim V +1$, $\dim P(\k)\ge \dim V +2$, and $\dim P(\k)$ is even.
   \item If $V \cong V\bidu$, then $\dim P(V) \ge  2\dim V +2$ and
   $$
   \sum_{W \in \Irr_1(H)} [P(V): W]
   $$
   is a positive even integer.
   }
\end{lem}
\begin{proof} (i) follows immediately from Lemma \ref{l:composition} and Corollary
   \ref{c:Sorder}.\\
   (ii) If $V \cong V\bidu$, $\dim P(V)$ is even, and hence the inequality follows. Moreover, the number of odd dimensional simple constituent of $P(V)$ including multiplicities
   must be even, or $\sum_{W \in \Irr_1(H)} [P(V): W]$ is even. By Lemma \ref{l:composition}, $[P(\k): V] = [P(V): \k_{\a\inv}] \ge 1$ where $\a$ is the distinguished group-like element of $H^*$. Therefore,  $\sum_{W \in \Irr_1(H)} [P(V): W] >0$.
\end{proof}

The cyclic group $\DD$ generated by the antipode  $S$ acts on $\Irr(H)$ by duality, i.e.
$S\cdot U \cong U\du$ for $U \in \Irr(H)$.
 The orbit $\DD_U$ of $U \in \Irr(H)$ is given by
$$
\DD_U=\{U, U^{\vee}, U\bidu ,U^{\vee\vee\vee}, \dots\}
$$
up to isomorphism.
It follows from Lemma \ref{l:4} that $|\DD_U| = 1,2,4$ or 8 for all simple $H$-modules $U$. By \eqref{eq:socle},
$$
P(U)^{\vee\vee\vee} \o \k_\a  \cong P( U\du)
$$
for $U \in \Irr(H)$. Thus,
\begin{equation} \label{eq:dimPdual}
\dim P(W) = \dim P(U) \quad \text{ for all } W \in \DD_U\,.
\end{equation}
Therefore, one has the equation
\begin{align}\label{eq:dimension}
\dim H=\sum_{U\in \mathrm{Irr}(H)}\dim U\cdot\dim P(U) = \sum_{U \in \RR}|\DD_U|\cdot \dim U\cdot\dim P(U)
\end{align}
where $\RR$ is a complete set of representatives of $\DD$-orbits.
\begin{lem}\label{l:step2}
 For any simple $H$-module $U$ such that $\dim U > 1$ and $|\DD_U|=1$,  $$\dim U \cdot \dim P(U) \ge 16.$$
\end{lem}
\begin{proof}
  If $|\DD_U|=1$, then $U \cong U\du$. In view of Lemma \ref{l:dual},  $\dim U \ge 3$. Since $U \cong U\bidu$,  $\dim P(U)$ is even by Corollary \ref{c:Sorder}. Thus, $\dim P(U) > \dim U$ for $\dim U$ odd. In particular, $U$ is not projective and hence $\dim P(U) \ge 2 \dim U$. Therefore,
  $$
  \dim U \cdot\dim P(U) \ge \dim U\cdot 2 \dim U \ge 18
  $$
  for any odd dimensional $U$. If $\dim U$ is even, then $\dim U \ge 4$ and so
  $$
   \dim U \cdot\dim P(U) \ge (\dim U)^2  \ge 16\,. \qedhere
  $$
\end{proof}
Following \cite{EG03}, we always have
\begin{equation}\label{eq:estimate}
 \dim U \cdot \dim P(U) \ge \dim P(\k)
\end{equation}
for all $U\in \Irr(H)$. This inequality together with \eqref{eq:dimension} imply
\begin{equation}\label{eq:numirr}
|\Irr(H)| \le \dim H/\dim P(\k)\,.
\end{equation}

\begin{lem}\label{l:classes}
   Suppose $H^*$ does not have any non-trivial skew primitive element. Then there exist $U, W \in \Irr(H)$ such that $\dim U, \dim W \ge 2$ and $\DD_U \ne \DD_W$.
\end{lem}
\begin{proof} Since $H^*$ does not have any non-trivial skew primitive element, $H^*$ is not pointed. In view of Lemma \ref{l:primitive}, $P(\k)$ admits a simple subquotient $V$ of dimension greater than one.  By Theorem \ref{t:Gorder}, $|G(H^*)| \le 2$.

   Assume the lemma is false.  Then $\DD_V$ is a complete set of non-isomorphic simple $H$-modules with dimension greater than 1. For simplicity, we let
   $$
   D_0=\dim P(\k),\, D=\dim P(V),\, d=\dim V \text{ and } m=|\DD_V|\,.
   $$
    By Lemma \ref{l:step1}, $D_0 \ge 4$ and $D \ge 5$. Equation \eqref{eq:dimension} becomes
  \begin{equation}\label{eq:dim}
     4p=|G(H^*)| D_0+ m d D\,.
  \end{equation}
  The $H$-module  $V\du \o P(V)$ is projective, and
  $\dim \Hom_H(V\du\o P(V), \k_\b)$
   is equal to the multiplicity of  $P(\k_\b)$ in an indecomposable
  decomposition of $V\du \o P(V)$ for $\b \in G(H)$. By \eqref{eq:iso},
   $$
  \dim \Hom_H(V\du\o P(V), \k_\b) = [V\o \k_\b: V]=\left\{\begin{array}{ll}
  1 & \text{if $V \cong V \o \k_\b$,}\\
  0 & \text{otherwise.}
  \end{array}\right.
  $$
  Therefore, by \eqref{eq:dimPdual} and a decomposition of $V\du \o P(V)$ into indecomposable $H$-modules, we find
   \begin{equation}\label{eq:dim2}
     dD=\dim ( V\du \o P(V))=n_0 D_0+ n_1  D
  \end{equation}
  where $n_0, n_1$ are integers satisfying $1\le n_0\le |G(H^*)|$ and $0\le n_1 < d$.
  Equations \eqref{eq:dim} and \eqref{eq:dim2} imply
 \begin{equation}\label{eq:dim3}
  (d-n_1)D=n_0 D_0 \quad\text{and}\quad 4p=(md+k(d-n_1))D,
  \end{equation}
 where $k=\frac{|G(H^*)|}{n_0} \in \BZ_+$.
  Since $m, k \ge 1$ and $d\ge 2$, $md+k(d-n_1)  \ge 3$. Equation \eqref{eq:dim3} implies $D=p$ and $md+k(d-n_1)=4$. In particular, $D$ is odd and so $V\not\cong V\bidu$. Thus $m \ge 4$, and hence $md+k(d-n_1)> 8$, a contradiction.
\end{proof}
Now, we can prove our key lemma in this section.
\begin{proof}[Proof of Lemma \ref{l:skewprimitive}] By duality, it suffices to show that $H^*$ has a non-trivial skew primitive element. Suppose $H^*$ has only trivial skew primitive elements. Then $H^*$ is not pointed, and hence, by Theorem \ref{t:Gorder}, $|G(H^*)|\le 2$. In view of \cite[Lemma 2.3]{EG03}, $P(\k)$ has a composition factor $V$ with $\dim V \ge 2$. By Lemma \ref{l:classes}, there exists a simple $H$-module $W$ such that $\dim W \ge 2$ and $\DD_W \ne \DD_V$. \medskip\\
(i) \emph{$|\DD_V|\le 2$ and $\dim P(V) \ge 2\dim V +2$}. Note that
$$
\dim W \dim P(W) \ge \dim P(\k) \ge 4.
$$
By  \eqref{eq:dimension} and Lemma \ref{l:step1},
  $$
  44 \ge \dim P(\k) + |\DD_V|\dim V \dim P(V)+\dim W \dim P(W) \ge  4 + |\DD_V|\cdot 2\cdot 5 + 4 \,.
  $$
  This implies $|\DD_V| < 4$ and hence $|\DD_V|\le 2$. In particular, $V \cong V\bidu$. Thus, the desired inequality follows from Lemma \ref{l:step1}(ii). \medskip\\
   (ii) \emph{$\dim V = 2$ and $\dim P(V)\ge 6$.} In view of (i), the  inequality follows from the first equality.

Assume $\dim V \ge 3$. By (i) and Lemma \ref{l:step1}, $\dim P(V) \ge 8$ and $\dim P(\k)\ge 6$. It follows from  \eqref{eq:estimate} that
   \begin{equation}\label{eq:1boundforW}
    \dim W  \dim P(W) \ge 6\,.
   \end{equation}
   We claim that $\dim W  \dim P(W) \ge 8$. This inequality obviously holds for $\dim W > 2$. For $\dim W =2$, the inequality \eqref{eq:1boundforW} implies  $\dim P(W) > \dim W$  and hence $\dim P(W) \ge 2\dim W$. Therefore,  $\dim W \dim P(W) \ge 8$.

   Now, we apply \eqref{eq:dimension} to obtain
  $$
 44 \ge \dim P(\k) + \dim V \dim P(V)+ |\DD_W|\dim W \dim P(W)\ge  6 + 3 \cdot 8 + |\DD_W| 8.
  $$
  The inequality forces $|\DD_W| =1$, and hence, by Lemma \ref{l:step2}, $\dim W \dim P(W) \ge 16$. However, this implies
  \begin{multline*}
  44 \ge \dim H \ge \dim P(\k) + \dim V  \dim P(V) + \dim W  \dim P(W)  \ge 6+ 24 + 16
  \end{multline*}
  which is absurd.\medskip\\
  (iii) $|\DD_V|=2$ by Lemma \ref{l:dual}, (i) and (ii). \medskip\\
   (iv) \emph{All composition factors of $P(V)$ are of dimension 1 or 2.}
  Suppose there exists a composition factor $U$ of $P(V)$ such that $\dim U \ge 3$. Then, by Lemma \ref{l:composition} and Corollary \ref{c:Sorder},
 $$
 \dim P(V),\, \dim P(U) \ge 8\,.
 $$
 However,
 \begin{multline*}
\dim H \ge \dim P(\k) + |\DD_V|\dim V\dim P(V) + \dim U \dim P(U) \\
\ge 4+2\cdot 2 \cdot 8 + 3\cdot 8 >44\,.
   \end{multline*}
 Therefore, all the composition factors of $P(V)$ are of dimension 1 or 2.\medskip\\
 (v) Let $M = \bigoplus_{\b \in G(H^*)} P(\k_\b)$.
By Lemma \ref{l:composition},
$$
[M: V]= \sum_{\b \in G(H^*)} [P(\k_\b):V] = \sum_{\b \in
G(H^*)} [P(V): \k_\b]\,.
$$
By (iv) and Lemma \ref{l:step1}(ii), $[M: V] \ge 2$.

Since $M$
is self-dual, $[M: V\du]=[M: V]$. Thus, we have
$$
\dim M \ge 2|G(H^*)|+2[M:V] \dim V \ge 2|G(H^*)|+ 8.
$$
If $|G(H^*)|=2$, then $\dim P(\k) = \dim M/2 \ge 6$ and so $\dim W
\dim P(W) \ge 6$. Since $\dim P(V) \ge 6$, we find
$$
\dim H \ge \dim M + 2 \dim V \dim P(V) + |\DD_W| 6 \ge 12+2 \cdot 2 \cdot 6 + |\DD_W|
6\,.
$$
This again forces $|\DD_W|=1$. It follows from Lemma \ref{l:step2}
that $\dim W \dim P(W) \ge 16$. However,
$$
\dim M + 2 \dim V \dim P(V) + \dim W \dim P(W) > 44\,.
$$
Therefore, $|G(H^*)|=1$. In this case, $\dim P(\k)=\dim M \ge 10$
and so  $\dim W \dim P(W) \ge 10$. By \eqref{eq:numirr}, $|\Irr(H)|
\le 4$. Therefore,  $|\DD_W|=1$ and so $\dim W \dim P(W) \ge 16$.
This implies
$$
\dim H \ge \dim P(\k) + 2 \dim V \dim P(V) + \dim W \dim P(W) \ge
50,
$$
a contradiction. This completes the proof of the lemma.
\end{proof}
\medskip
\noindent{\bf Acknowledgment:}
The authors would like to thank Zhongzhu Lin for his suggestion on the proof of Lemma \ref{l:primitive}, and
Michael Hilgemann for his comments on the first draft of this paper.

%\bibliographystyle{amsalpha}
%\bibliography{mybibl}
%\end{document}
\providecommand{\bysame}{\leavevmode\hbox to3em{\hrulefill}\thinspace}
\providecommand{\MR}{\relax\ifhmode\unskip\space\fi MR }
\providecommand{\MRhref}[2]{%
  \href{http://www.ams.org/mathscinet-getitem?mr=#1}{#2}
}
\providecommand{\href}[2]{#2}

\end{document}